\documentclass[11pt,reqno]{amsproc}
\usepackage[utf8]{inputenc}
\usepackage[english]{babel} 

\usepackage{amsmath}
\usepackage{amssymb}
\usepackage{amsthm}

\usepackage{stmaryrd}

\usepackage{cite}

\usepackage[matrix,arrow,curve]{xy}

\begin{document}
\righthyphenmin=2

\renewcommand{\refname}{References}
\renewcommand{\bibname}{Bibliography}
\renewcommand{\proofname}{Proof}

\newtheorem{lm}{Lemma}
\newtheorem*{lm*}{Lemma}
\newtheorem{tm}{Theorem}
\newtheorem*{tm*}{Theorem}
\newtheorem*{prop}{Proposition}
\newtheorem*{mtm}{Main Theorem}
\newtheorem*{atm}{Another Presentation}
\newtheorem{cl}{Corollary}
\newtheorem{mcl}{Corollary}
\newtheorem*{cl*}{Corollary}
\theoremstyle{definition}
\newtheorem*{df}{Definition}
\theoremstyle{remark}
\newtheorem*{rk}{Remark}

\numberwithin{lm}{section}
\numberwithin{cl}{lm}

\newcommand{\Max}{\mathop{\mathrm{Max}}\nolimits}
\newcommand{\proj}{\mathop{\mathrm{proj}}\nolimits}
\newcommand{\Card}{\mathop{\mathrm{Card}}\nolimits}
\newcommand{\Ker}{\mathop{\mathrm{Ker}}\nolimits}
\newcommand{\Cent}{\mathop{\mathrm{Cent}}\nolimits}
\newcommand{\E}{{\mathrm{E}}}
\newcommand{\St}{\mathop{\mathrm{St}}\nolimits}
\newcommand{\Sp}{\mathop{\mathrm{Sp}}\nolimits}
\newcommand{\Ep}{\mathop{\mathrm{Ep}}\nolimits}
\newcommand{\GL}{\mathop{\mathrm{GL}}\nolimits}
\newcommand{\Kt}{\mathop{\mathrm{K_2}}\nolimits}
\newcommand{\Ko}{\mathop{\mathrm{K_1}}\nolimits}
\newcommand{\Ho}{\mathop{\mathrm{H_1}}\nolimits}
\newcommand{\Ht}{\mathop{\mathrm{H_2}}\nolimits}
\newcommand{\epi}{\twoheadrightarrow}
\newcommand{\sgn}{\mathrm{sgn}}
\newcommand{\eps}[1]{\varepsilon_{#1}}
\newcommand{\sign}[1]{\mathrm{sign}(#1)}
\newcommand{\lan}{\langle}
\newcommand{\ran}{\rangle}
\newcommand{\inv}{^{-1}}
\newcommand{\ur}[1]{\!\,^{(#1)}U_1}
\newcommand{\ps}[1]{\!\,^{(#1)}\!P_1}
\newcommand{\ls}[1]{\!\,^{(#1)}\!L_1}

\renewcommand{\labelenumi}{\theenumi{\rm)}}
\renewcommand{\theenumi}{\alph{enumi}}

\author{Andrei Lavrenov}

\title{A local--global principle for symplectic $\Kt$}

\date{}

\begin{abstract}
We prove that an element of the symplectic Steinberg group is trivial if and only if its image under any maximal localisation homomorphism is trivial.
\end{abstract}

\maketitle

\section*{Introduction}

The main goal of the present paper is to establish a symplectic $\mathrm K_2$-analogue of Quillen's Patching Theorem which is the key ingredient in his solution of Serre's problem (see~\cite{n13})\footnote{The author acknowledges financial support from Russian Science Foundation grant 14--11--00297.}.

\begin{tm*}[Quillen]
Let $R$ be a commutative ring, and $P$ be a finitely generated projective module over the polynomial ring $R[t_1,\ldots,t_n]$. Then $P$ is extended from $R$ {\rm (}i.e., there exists a projective $R$-module $Q$ such that $P\cong R[t_1,\ldots,t_n]\otimes Q${\rm )} if and only if $P_{\mathfrak m}$ is extended from $R_{\mathfrak m}$ for every maximal ideal $\mathfrak m$ of $R$.
\end{tm*}

Later Suslin proved his $\Ko$-analogue of Serre's problem with the use of a similar statement concerning elementary matrices (see~\cite{n11}).

\begin{tm*}[Suslin]
Let $n\geq3$ and $g\in\GL_n(R[t],\,tR[t])$. Then $g\in\E_n(R[t])$ if and only if $g_{\mathfrak m}\in\E_n(R_{\mathfrak m}[t])$ for every maximal ideal $\mathfrak m$ of $R$.
\end{tm*}

Subsequently, $\Ko$-analogues for Chevalley groups~\cite{n12,n5,n1} were obtained. Presently, $\Ko$-analogue is known in much larger generality, namely, for isotropic reductive groups~\cite{n9} and in the framework of Stepanov's universal localisation~\cite{n10}. As opposed to that, $\Kt$-analogues are proven only for $\GL_n$~\cite{n14} and, recently, for Chevalley groups of type $\mathrm E_l$~\cite{n8}.

The Main Theorem of the present paper is the following $\Kt$-analogue to the local-global principle for $\Sp_{2n}$.

\begin{mtm}
Let $R$ be an arbitrary commutative ring {\rm(}with 1{\rm)}, $n\geq3$, and $g\in\St\!\Sp_{2n}(R[t],\,tR[t])$. Then $g=1$ {\rm (}corr., lies in the image of $\St\!\Sp_{2n-2}(R[t])${\rm )} if and only if $g_{\mathfrak m}=1\in\St\!\Sp_{2n}(R_{\mathfrak m}[t])$ {\rm (}corr., lies in the image of $\St\!\Sp_{2n-2}(R_{\mathfrak m}[t])${\rm )} for every maximal ideal $\mathfrak m$ of $R$.
\end{mtm}

The paper is organised as follows. In the first section we recall results of~\cite{l1}, where a ``basis-free'' presentation of the symplectic Steinberg group id given. We make an essential use of these results in the Section~4. In the next section we show that three possible definitions of the relative symplectic Steinberg group coincide. One of them is the ``correct'' definition, and the other two are used in our proof. In the Section~3 we establish the local-global principle modulo the main technical lemma. Finally, in the Section~4, we give a proof to this lemma, namely, construct a symplectic analogue of Tulenbaev map.

\section{Absolute Steinberg groups}

In the present paper $R$ always denotes an arbitrary associative commutative unital ring, $R^{2n}$ is the free right $R$-module, we number its basis as follows: $e_{-n}$, $\ldots$, $e_{-1}$, $e_1$, $\ldots$, $e_n$, $n\geq3$. {\it The symplectic group} $\Sp_{2n}(R)$ is the group of automorphisms of $R^{2n}$ preserving the standard symplectic form $\lan\,,\,\ran$, where $\lan e_i,\,e_{-i}\ran=1,\ i>0$. We denote {\it the elementary symplectic transvections} by 
$$T_{ij}(a)=1+e_{ij}\cdot a-e_{-j,-i}\cdot a\cdot\sign i\cdot\sign j,\quad T_{i,-i}(a)=1+e_{i,-i}\cdot a\cdot\sign i,$$
 where $a\in R$, $i$, $j\in\{-n,$ $\ldots,$ $-1,$ $1,$ $\ldots,$ $n\}$, $i\not\in\{\pm j\}$, $e_{ij}$ is a matrix unit. They generate {\it the elementary symplectic group} $\Ep_{2n}(R)$.

\begin{df}
The {\it symplectic Steinberg group} $\St\!\Sp_{2n}(R)$ is generated by the formal symbols $X_{ij}(a)$ for $i\neq j$, $a\in R$ subject to the Steinberg relations
\setcounter{equation}{-1}
\renewcommand{\theequation}{S\arabic{equation}}
\begin{align}
&X_{ij}(a)=X_{-j,-i}(-a\cdot\sign i\sign j),\\
&X_{ij}(a)X_{ij}(b)=X_{ij}(a+b),\\
&[X_{ij}(a),\,X_{hk}(b)]=1,\text{ for }h\not\in\{j,-i\},\ k\not\in\{i,-j\},\\
&[X_{ij}(a),\,X_{jk}(b)]=X_{ik}(ab),\text{ for }i\not\in\{-j,-k\},\ j\neq-k,\\
&[X_{i,-i}(a),\,X_{-i,j}(b)]=X_{ij}(ab\cdot\sign i)X_{-j,j}(-ab^2),\\
&[X_{ij}(a),\,X_{j,-i}(b)]=X_{i,-i}(2\,ab\cdot\sign i).
\end{align}
\end{df}

There is a natural projection $\phi\colon\St\!\Sp_{2n}(R)\epi\Ep_{2n}(R)$ sending the generators $X_{ij}(a)$ to $T_{ij}(a)$. In other words, the Steinberg relations hold for the elementary symplectic transvections.

To define the elementary symplectic group $\Ep_{2n}(R)$ instead of $T_{ij}(a)$ one can use ``basis-independent'' ESD-transformations $T(u,\,v,\,a)$ defined by $$w\mapsto w+u(\lan v,\,w\ran+a\lan u,\,w\ran)+v\lan u,\,w\ran$$
as a set of generators. Here $u$, $v\in R^{2n}$, $\lan u,\,v\ran=0$, $a\in R$. See section~1 of~\cite{l1} for details. On the level of Steinberg groups, this idea leads to the following presentation, inspired by van der Kallen's paper~\cite{n3}. It is the main theorem of~\cite{l1}.
\begin{tm*}
The symplectic Steinberg group $\St\!\Sp_{2n}(R)$ can be defined by the set of generators
$$
\big\{[u,\,v,\,a]\,\big|\ u\in\Ep_{2n}(R)e_1,\,v\in R^{2n},\ \ \lan u,\,v\ran=0,\ a\in R\big\}
$$
and relations
\setcounter{equation}{0}
\renewcommand{\theequation}{K\arabic{equation}}
\begin{align}
&[u,\,v_1,\,a_1][u,\,v_2,\,a_2]=[u,\,v_1+v_2,\,a_1+a_2+\lan v_1,\,v_2\ran],\\
&[u_1,\,u_2b,\,0]=[u_2,\,u_1b,\,0]\,\text{ for any }\, b\in R,\\
&[u',\,v',\,a'][u,\,v,\,a][u',\,v',\,a']\inv=[T(u',\,v',\,a')u,\,T(u',\,v',\,a')v,\,a],
\end{align}
For the usual generators of the symplectic Steinberg group the following identities hold 
$$
X_{ij}(a)=[e_i,\,e_{-j}\cdot a\cdot\sign{-j},\,0]\ \text{for $j\neq-i$},\quad X_{i,-i}(a)=[e_i,\,0,\,a],
$$
and, moreover, $\phi$ sends $[u,\,v,\,a]$ to $T(u,\,v,\,a)$. Furthermore, the following relations are automatically satisfied
\begin{align}
&[u,\,ua,\,0]=[u,\,0,\,2a],\\
&[ub,\,0,\,a]=[u,\,0,\,ab^2],\\
&[u+v,\,0,\,a]=[u,\,0,\,a][v,\,0,\,a][v,\,ua,\,0]\,\text{ for }\lan u,\,v\ran=0.
\end{align}
\end{tm*}

More precisely, we need elements $X(u,\,v,\,a)$ in $\St\!\Sp_{2n}(R)$ defined in~\cite{l1} for arbitrary $u$, $v\in R^{2n}$ and $a\in R$ (see the definition of $X(u,\,0,\,a)$ in section~4, then the definition of $X(u,\,v,\,0)$ on p.~3775 and the general case on p.~3777, or an overview in section~1). For $u\in\Ep_{2n}(R)$ these elements coincide with the generators $[u,\,v,\,a]$ above (see the last section of~\cite{l1} where the isomorphism between two presentations is described). Below we list their properties, proven in~\cite{l1}.

\begin{lm}
\label{xlist}
For all $u$, $v\in R^{2n}$, $\lan u,\,v\ran=0$, $a$, $b\in R$, $g\in\St\!\Sp_{2n}(R)$ one has
\setcounter{equation}{-1}
\renewcommand{\theequation}{X\arabic{equation}}
\begin{align}
&\phi\big(X(u,\,v,\,a)\big)=T(u,\,v,\,a),\\
&g\,X(u,\,v,\,a)g\inv=X(\phi(g)u,\,\phi(g)v,\,a),\\
&X(ub,\,0,\,a)=X(u,\,0,\,b^2a),\\
&X(u,\,0,\,a)X(u,\,0,\,b)=X(u,\,0,\,a+b),\\
&X(u,\,v,\,0)=X(v,\,u,\,0),\\
&X(ua,\,ub,\,0)=X(u,\,0,\,2ab),\\
&X(u+v,\,0,\,1)=X(u,\,0,\,1)X(v,\,0,\,1)X(u,\,v,\,0),\\
&X(u,\,v,\,a)=X(u,\,v,\,0)X(u,\,0,\,a).
\end{align}
\end{lm}
\begin{proof}
For X0 see section~1 of~\cite{l1}, properties X1--X4 are proven in Lemmas~26, 29, 31, 33, 34 of~\cite{l1}, X6 and X7 are actually definitions (see pp.~3775 and 3777 of~\cite{l1}). To get X5 use X6, then X2 and X3.
\end{proof}

We also need to use a few more properties of $X(u,\,v,\,a)$ listed in Lemma~\ref{xlist2} below. To prove them, we pass to another type of elements $Y(u,\,v,\,a)$ defined in~\cite{l1}. As soon as Lemma~\ref{xlist2} is proven, we do not need $Y$'s any more in this paper and use only $X$'s.

For $u\in R^{2n}$ having a pair of zeros in symmetric positions, i.e., such that $u_i=u_{-i}=0$ for some $i$, and $v\in R^{2n}$, such that $\lan u,\,v\ran=0$, $a\in R$ also elements $Y_{(i)}(u,\,v,\,a)$ are defined. Their definition is given in three steps (see definitions and remarks on pp.~3762, 3764 and 3766 of~\cite{l1}). If, in addition, $u$ has another symmetric pair of zeros, i.e., $u_i=u_{-i}=u_j=u_{-j}=0$ for some $i\neq\pm j$, then $Y_{(i)}(u,\,v,\,a)=Y_{(j)}(u,\,v,\,a)$ (see Y1 below) and in this situation we omit the index and denote this element by $Y(u,\,v,\,a)$. We use the following properties of these elements, proven in~\cite{l1}.

\begin{lm}
\label{ylist}
For a fixed index $\,i$ and any $j\neq\pm i$, vectors $u$, $v$, $v'$, $w$, $w'$, $q$, $q'$, $r$, $r'$ $s\in R^{2n}$, such that $u_i=u_{-i}=0$, $\lan u,\,v\ran=\lan u,\,v'\ran=0$, $\lan w,\,e_i\ran=\lan w',\,e_i\ran=0$, $v'_i=v'_{-i}=q_i=q_{-i}=r_i=r_{-i}=r_j=r_{-j}=0$, $q'=e_iq'_i+e_{-i}q'_{-i}$, $r'=e_jr'_j+e_{-j}r'_{-j}$, $s=e_is_i+e_{-i}s_{-i}$, and elements $a$, $a'\in R$, one has
\setcounter{equation}{-1}
\renewcommand{\theequation}{Y\arabic{equation}}
\begin{align}
&\phi\big(Y_{(i)}(u,\,v,\,a)\big)=T(u,\,v,\,a),\\
&Y_{(i)}(u,\,v,\,a)=Y_{(j)}(u,\,v,\,a)\,\text{ if also }\,u_j=u_{-j}=0,\\
&Y(e_i,\,w,\,a)Y(e_i,\,w',\,a')=Y(e_i,\,w+w',\,a+a'+\lan w,\,w'\ran),\\
&Y(e_i,\,w,\,a)=X(e_i,\,w,\,a),\\
&\!\begin{aligned}
Y_{(i)}(u,\,v',\,a)=[Y(e_i,\,u,\,0)&,\,Y(e_{-i},\,v'\,\sign i,\,a)]\cdot\\
\cdot Y&(e_i,\,ua\,\sign{-i},\,0),
\end{aligned}\\
&\!\begin{aligned}
Y_{(i)}(u,\,v,\,a)=Y_{(i)}(u,\,v-e_iv_i-e_{-i}v_{-i}&,\,a-v_iv_{-i}\,\sign i)\cdot\\
\cdot&Y(e_i,\,uv_i,\,0)Y(e_{-i},\,uv_{-i},\,0),
\end{aligned}\\
&X(q+q',\,0,\,a)=Y_{(i)}(q,\,0,\,a)Y(q',\,0,\,a)Y(q',\,qa,\,0),\\
&Y(r,\,s,\,0)Y(r',\,s,\,0)=Y_{(i)}(r+r',\,s,\,0).
\end{align}
\end{lm}
\begin{proof}
For Y0 see section 1 of~\cite{l1}, for Y1 and Y2 see Lemmas~20 and 9 of~\cite{l1}, for Y3 use Lemma~36 of~\cite{l1}, Y6, Lemma~\ref{xlist}(X7) and Y2, Y4--Y6 are actually definitions (pp.~3764, 3766, 3768 and remark on p.~3770 of~\cite{l1}), Y7 is Lemma~28 of~\cite{l1}.
\end{proof}

Let us establish some further properties of $Y_{(i)}(u,\,v,\,a)$.

\begin{lm}
For an index $\,i$ and any vectors $u$, $v$, $w\in R^{2n}$, such that $u_i=u_{-i}=0$, $\lan u,\,v\ran=\lan u,\,w\ran=0$, elements $a$, $b\in R$, one has
\setcounter{equation}{7}
\renewcommand{\theequation}{Y\arabic{equation}}
\begin{align}
&Y_{(i)}(u,\,v,\,a+b)=Y_{(i)}(u,\,v,\,a)Y_{(i)}(u,\,0,\,b),\\
&Y_{(i)}(u,\,v,\,a)=Y_{(i)}(u,\,v-e_iv_i-e_{-i}v_{-i},\,a)Y_{(i)}(u,\,e_iv_i+e_{-i}v_{-i},\,0),\\
&X(u+v,\,0,\,a)=X(u,\,0,\,a)X(v,\,0,\,a)Y_{(i)}(u,\,va,\,0),\\
&Y_{(i)}(u,\,va,\,0)=Y_{(i)}(v,\,ua,\,0)\,\text{ if also }\,v_i=v_{-i}=0,\\
&Y_{(i)}(u,\,v,\,a)=X(u,\,v,\,a),\\
&Y_{(i)}(u,\,v,\,0)Y_{(i)}(u,\,w,\,0)=Y_{(i)}(u,\,v+w,\,\lan v,\,w\ran).
\end{align}
\end{lm}
\begin{proof}
We prove Y8 and Y9 together. 

First, assume that $v_i=v_{-i}=0$. In this assumption one can get Y8 repeating the proof of Lemma~21 of~\cite{l1}. Use Y2 and Y4 to decompose elements and Y3 and X1 to show, that some of them commute (instead of complicated arguments in the original proof). Next, use this result to obtain Y9, more precisely, repeating the proof of Lemma~22 of~\cite{l1} use it instead of Lemma~21. Now, Y8 follows from Y9 in full generality. One needs that $Y_{(i)}(u,\,0,\,b)=X(u,\,0,\,b)$ by Y6, and commutes with $Y_{(i)}(u,\,e_iv_i+e_{-i}v_{-i},\,0)$ by X1.

To obtain Y10--Y13 we also proceed in several steps.

First, consider Y10 and assume that $v_i=v_{-i}=0$. Then, repeat the proof of Lemma~23 of~\cite{l1} interchanging the roles of $u$ and $v$ and using Y3 and X1 instead of the arguments presented there. Obviously, by X1, $X(u,\,0,\,a)$ and $X(v,\,0,\,a)$ commute. Using this fact we get Y11. Then, we can get Y12 in the same assumptions on $v$. For $a=0$ it follows from Y10 and X6, for the general case use Y8 and X7.

Next, consider Y13 and assume that $v_i=v_{-i}=w_i=w_{-i}=0$. By Y4, we have
\begin{multline*}
Y(u,\,v+w,\,\lan v,\,w\ran)=\\=[Y(e_i,\,u,\,0),\,Y(e_{-i},\,(v+w)\,\sign i,\,\lan v,\,w\ran)]Y(e_{-i},\,u\lan v,\,w\ran\sign{-i}).
\end{multline*}
Now, we get
\begin{multline*}
Y(e_{-i},\,(v+w)\,\sign i,\,\lan v,\,w\ran)=Y(e_{-i},\,v\,\sign i,\,0)Y(e_{-i},\,w\,\sign i,\,0)
\end{multline*}
by Y2 and use an identity $[a,\,bc]=[a,\,b]\cdot[a,\,c]\cdot[[c,\,a],\,b]$ for $a=Y(e_i,\,u,\,0)$, $b=Y(e_{-i},\,v\,\sign i,\,0)$, $c=Y(e_{-i},\,w\,\sign i,\,0)$. As above, $[a,\,b]=Y_{(i)}(u,\,v,\,0)$ and $[a,\,c]=Y_{(i)}(u,\,w,\,0)$. One can check, that
$$[[c,\,a],\,b]=Y(e_{-i},\,-u\lan w,\,v\ran\sign i,\,0)$$
with the use of Y3, X1 and Y2.

Now, we prove Y13 for arbitrary $v$, $w$ and $u$ having two pairs of zeros, i.e., such that $u_j=u_{-j}=0$ for $j\neq\pm i$ also. Decompose $v=\tilde v+v'$, where $v'=e_iv_i+e_{-i}v_{-i}$, and similarly for $w$. Use Y9, then Y8, then Y6 and X1 to change the order of factors, then Y8 again to obtain
$$
Y(u,\,v+w,\,\lan v,\,w\ran)=Y(u,\,\tilde v+\tilde w,\,\lan \tilde v,\,\tilde w\ran)Y(u,\,v'+w',\,\lan v',\,w'\ran).
$$
For each factor we can use the previous steps. To reorder the factors in the result use that $Y(u,\,v',\,0)=X(u,\,v',\,0)$ (we already have Y12 for this situation) and X1. Then use Y9 again.

Next, we establish Y10 in full generality. Decompose $v=\tilde v+v'$ as above and use Y6 to obtain
$$
X(u+v,\,0,\,a)=X(u+\tilde v,\,0,\,a)X(v',\,0,\,a)Y(v',\,(u+\tilde v)a,\,0).
$$
For the first factor we can already use Y10, and for the last one we can already use Y13. Reordering factors by X1 we get 
$$
X(u+v,\,0,\,a)=X(u,\,0,\,a)Y_{(i)}(u,\,\tilde va,\,0)Y(v',\,ua,\,0)X(v,\,0,\,a)
$$
with the use of Y6. Decompose $u=\tilde u+u'$, where $u'=e_ju_j+e_{-j}u_{-j}$, then decompose $Y(v',\,ua,\,0)$ by Y9, apply Y11 to each factor and use Y7 to get
\begin{multline*}
Y(v',\,ua,\,0)=Y(v',\,\tilde ua,\,0)Y(v',\,u'a,\,0)=\\=Y_{(j)}(\tilde u,\,v'a,\,0)Y_{(k)}(u',\,v'a,\,0)=Y(u,\,v'a,\,0).
\end{multline*}
Now, we are done by Y9. Proceeding as above, we get Y12 in full generality.

Finally, consider Y13. As above, decompose $v=\tilde v+v'$ and $w=\tilde w+w'$, and get
$$
Y_{(i)}(u,\,v+w,\,\lan v,\,w\ran)=Y_{(i)}(u,\,\tilde v+\tilde w,\,\lan \tilde v,\,\tilde w\ran)Y_{(i)}(u,\,v'+w',\,\lan v',\,w'\ran).
$$
For the first factor Y13 holds by previous steps, for the second one use Y5, then Y8 and Y2. Then change the order of factors. To interchange positions of  $Y(e_i,\,uw_i,\,0)$ and $Y(e_{-i},\,uv_{-i},\,0)$ we need to plug in an extra commutator equal to
$$
Y_{(i)}(-uw_i,\,-uv_{-i}\,\sign i,\,0)=X(u,\,0,\,2w_iv_{-i}\,\sign i)
$$
by Y12 and X5. With the use of Y5 and Y8 one gets
$$
Y_{(i)}(u,\,v'+w',\,\lan v',\,w'\ran)=Y_{(i)}(u,\,v',\,0)Y_{(i)}(u,\,w',\,0).
$$
It remains to change the order of factors and use Y9.
\end{proof}

In the following lemma we collect some new relations among X's which we need in the sequel.

\begin{lm}
\label{xlist2}
Consider $u$, $v$, $w\in R^{2n}$, $a$, $b\in R$, such that $\lan u,\,v\ran=\lan u,\,w\ran=0$. Assume either that $u_i=u_{-i}=0$ or that $v_i=v_{-i}=w_i=w_{-i}=0$. Then one has
\setcounter{equation}{7}
\renewcommand{\theequation}{X\arabic{equation}}
\begin{align}
&X(u+vr,\,0,\,a)=X(u,\,0,\,a)X(v,\,0,\,r^2a)X(u,\,vra,\,0),\\
&X(u,\,va,\,0)=X(v,\,ua,\,0),\\
&X(u,\,v,\,a)X(u,\,w,\,b)=X(u,\,v+w,\,a+b+\lan v,\,w\ran).
\end{align}
\end{lm}
\begin{proof}
First, assume $u_i=u_{-i}=0$. Then X8 follows from Y10, X2 and Y12. Denote $u'=e_ju_j+e_{-j}u_{-j}$ and $v'=e_iv_i+e_{-i}v_{-i}$ for some $j\neq\pm i$ and decompose $u=\tilde u+u'$, $v=\tilde v+v'$. Then
$$
X(u,\,va,\,0)=Y_{(i)}(u,\,\tilde va,\,0)Y(\tilde u,\,v'a,\,0)Y(u',\,v'a,\,0)
$$
by Y9 and Y7. Now apply Y11 twice to each factor
$$
X(u,\,va,\,0)=Y_{(i)}(ua,\,\tilde v,\,0)Y(\tilde ua,\,v',\,0)Y(u'a,\,v',\,0).
$$
Next, X10 for $a=b=0$ follows from Y12 and Y13, and the general case from Y8, X7 and X3.

Now, consider the second case $v_i=v_{-i}=w_i=w_{-i}=0$. For X8 use the previous step (X8 and X9). It remains to prove X10. Assume $a=b=0$ (the case of arbitrary $a$ and $b$ may be treated as in the previous step). Denote $u'=e_iu_i+e_{-i}u_{-i}$ and $\tilde u=u-u'$ (previously we took $\pm j$-th components instead). Using X4, Y12 and Y9 we get
$$
X(u,\,v,\,0)=X(\tilde u,\,v,\,0)X(u',\,v,\,0)
$$
and similarly for $w$. Using X1 and the previous case (X10) we can compute
$$
[X(u',\,v,\,0),\,X(\tilde u,\,w,\,0)]=X(\tilde u,\,u'\lan v,\,w\ran,\,0).
$$
Thus, reordering factors and using previous case we get
\begin{multline*}
X(u,\,v,\,0)X(u,\,w,\,0)=\\=X(\tilde u,\,v+w,\,\lan v,\,w\ran)X(u',\,v+w,\,\lan v,\,w\ran)X(\tilde u,\,u'\lan v,\,w\ran,\,0).
\end{multline*}
Now, we can finish the proof with the use of X7, X1, X4, Y12, Y9 and Y6.
\end{proof}

At the end of this section we establish one more relation automatically satisfied by the generators $[u,\,v,\,a]$ from {\it Another presentation}. We need it later to get a map from the relative symplectic Steinberg group to the absolute one. First, observe that for $v\in R^{2n}$ such that $v_{-1}=0$ (equiv., $\lan v,\,e_1\ran=0$) holds
$$
X(e_1+vr,\,0,\,a)=X(e_1,\,0,\,a)X(v,\,0,\,r^2a)X(e_1,\,vra,\,0)
$$
by X8. Next, take $g\in\St\!\Sp_{2n}(R)$ and denote $u=\phi(g)e_1$ and $w=\phi(g) v$. Conjugate the above identity by $g$ and use X1 to get
$$
X(u+wr,\,0,\,a)=X(u,\,0,\,a)X(w,\,0,\,r^2a)X(u,\,wra,\,0).
$$
The above identity holds for any $u\in\Ep_{2n}(R)e_1$ and $w$ orthogonal to it. Finally, take $(u,\,w)\in\Ep_{2n}(R)(e_1,\,e_2)$. Then $u+wr$ also lies in $\Ep_{2n}(R)e_1$ and we can replace X's above by the generators from the {\it Another presentation}.
\begin{lm}
For $(u,\,w)\in\Ep_{2n}(R)(e_1,\,e_2)$ and $a$, $r\in R$ the following identity holds
\renewcommand{\theequation}{K\arabic{equation}}
\setcounter{equation}{6}
\begin{align}
&[u+wr,\,0,\,a]=[u,\,0,\,a][w,\,0,\,ar^2][u,\,war,\,0].
\end{align}
\end{lm}

\section{Relative Steinberg groups}

In the present section we give three definitions of a relative symplectic Steinberg group.

For our purposes we can concentrate on splitting ideals $I\trianglelefteq R$, i.e., those ideals for which the natural projection $\rho\colon R\epi R/I$ splits. Obviously, $tR[t]\trianglelefteq R[t]$ is a splitting ideal.

We show that for splitting ideals all three definitions of relative symplectic Steinberg group coincide.

The correct approach to relative Steinberg groups is described in~\cite{n4,n7,n8}. But for splitting ideals we can define it in the following naive way. Afterwards, we show that (for splitting ideals) it coincides with the usual one.

\begin{df}
Let $I\trianglelefteq R$ be  a {\it splitting} ideal. Define {\it the relative symplectic Steinberg group} $\St\!\Sp_{2n}(R,\,I)=\Ker\big(\rho^*\colon \St\!\Sp_{2n}(R)\epi\St\!\Sp_{2n}(R/I)\big)$.
\end{df}

Obviously, $\Ker(\rho_*)$ coincides with the normal subgroup of $\St\!\Sp_{2n}(R)$ generated by $\{X_{ij}(a)\mid a\in I\}$. This is tantamount to saying that applying $\rho^*$ is the same as forcing an additional relation $X_{ij}(a)=1,\ a\in I$, in $\St\!\Sp_{2n}(R)$.

The next definition is a symplectic version of the Keune--Loday presentation in the linear case (see~\cite{n4,n7}). It is a relative version of the definition via Steinberg relations.

For a group $G$ acting on a group $H$ on the left, we will denote the image of $h\in H$ under the homomorphism corresponding to the element $g\in G$ by $\!\,^gh$, the element $\!\,^gh\cdot h\inv$ by $\llbracket g,\,h]$ and the element $h\cdot\!\,^gh\inv$ by $[h,\,g\rrbracket$.

\begin{df}
Let {\it the Keune--Loday relative symplectic Steinberg group} $\St\!\Sp^{\rm KL}_{2l}(R,\,I)$ be a group with the action of the (absolute) Steinberg group $\St\!\Sp_{2l}(R)$ defined by the set of relative generators $Y_{ij}(a), i\neq j,\ a\in I$, subject to the relations
\setcounter{equation}{-1}
\renewcommand{\theequation}{KL\arabic{equation}}
\begin{align}
&Y_{ij}(a)=Y_{-j,-i}(-a\cdot\sign i\cdot\sign j),\\
&Y_{ij}(a)Y_{ij}(b)=Y_{ij}(a+b),\\
&\llbracket X_{ij}(r),\,Y_{hk}(a)]=1,\text{ for }h\not\in\{j,-i\},\ k\not\in\{i,-j\},\\
&\llbracket X_{ij}(r),\,Y_{jk}(a)]=Y_{ik}(ra),\text{ for }i\not\in\{-j,-k\},\ j\neq-k,\\
&\llbracket X_{i,-i}(r),\,Y_{-i,j}(a)]=Y_{ij}(ra\cdot\sign i)Y_{-j,j}(-ra^2),\\
&[Y_{i,-i}(a),\,X_{-i,j}(r)\rrbracket=Y_{ij}(ar\cdot\sign i)Y_{-j,j}(-ar^2),\\
&\llbracket X_{ij}(r),\,Y_{j,-i}(a)]=X_{i,-i}(2\,ra\cdot\sign i),\\
&\!\,^{X_{ij}(a)}\Big(\!\,^{X_{hk}(r)}Y_{st}(b)\Big)=\!\,^{Y_{ij}(a)}\Big(\!\,^{X_{hk}(r)}Y_{st}(b)\Big).
\end{align}
In other words, we consider a free group generated by symbols $(g,\,x)=\!\,^gx$ where $g$ is from the absolute Steinberg group and $x$ is from the set of relative generators, $\St\!\Sp_{2l}(R)$ naturally acts on this free group via $\!\,^f(g,\,x)=(fg,\,x)$ and then we define the relative symplectic Steinberg group as the quotient of the above free group modulo equivariant normal subgroup generated by KL0--KL7.
\end{df}

There is an obvious equivariant mapping from $\St\!\Sp^{\rm KL}_{2l}(R,\,I)$ to $\St\!\Sp_{2l}(R)$ sending $Y_{ij}(a)$ to $X_{ij}(a)$, and its image is the normal subgroup generated by $\{X_{ij}(a)\mid a\in I\}$, i.e., coincides with $\Ker\big(\St\!\Sp_{2n}(R)\epi\St\!\Sp_{2n}(R/I)\big)$.

\begin{lm}
Let $I\trianglelefteq R$ be a splitting ideal. Then the natural map $$\iota\colon\St\!\Sp^{\rm KL}_{2l}(R,\,I)\rightarrow\St\!\Sp_{2l}(R)$$ is injective. In other words, $\St\!\Sp^{\rm KL}_{2l}(R,\,I)=\St\!\Sp_{2l}(R,\,I)$.
\end{lm}

\begin{rm}
The proof is actually the same as in the linear case (see~\cite{n4,n7} and~\cite{n8} for the simply-laced case).
\end{rm}

\begin{proof}
Denote by $\rho\colon R\epi R/I$ the natural projection and by $\sigma\colon R/I\rightarrow R$ its splitting. Then $\St\!\Sp_{2l}(R/I)$ acts on $\St\!\Sp^{\rm KL}_{2l}(R,\,I)$ via $\sigma^*$ and one can consider the semi-direct product $\St\!\Sp^{\rm KL}_{2l}(R,\,I)\leftthreetimes\St\!\Sp_{2l}(R/I)$ which maps to $\St\!\Sp_{2l}(R)$ via $\iota\leftthreetimes\sigma^*$
$$
\St\!\Sp^{\rm KL}_{2l}(R,\,I)\leftthreetimes\St\!\Sp_{2l}(R/I)\rightarrow\St\!\Sp_{2l}(R),\quad (x,\,y)\mapsto\iota(x)\cdot\sigma^*(y).
$$
We construct an inverse map 
$$\psi\colon\St\!\Sp_{2l}(R)\rightarrow\St\!\Sp^{\rm KL}_{2l}(R,\,I)\leftthreetimes\St\!\Sp_{2l}(R/I),$$
sending $$X_{ij}(r)\mapsto\big(Y_{ij}(r-\sigma\rho(r)),\,X_{ij}(\rho(r))\big).$$
Obviously, the fact that $\iota\leftthreetimes\sigma^*$ is an isomorphism implies that $\iota$ is injective.
 
To check that $\psi$ is well-defined one has to verify relations S0--S5 for the images of the generators. Consider, say, S4. We will show that the images of 
$$
X_{i,-i}(a)X_{-i,j}(b)X_{i,-i}(-a)\text{ and }X_{ij}(ab\cdot\sign i)X_{-j,j}(-ab^2)X_{-i,j}(b)
$$ 
under $\psi$ coincide. Indeed,
\begin{multline*}
\begin{aligned}
&\psi\Big(X_{i,-i}(a)X_{-i,j}(b)X_{i,-i}(-a)\Big)=\\
&=\Big(Y_{i,-i}(a-\sigma\rho(a))\,^{X_{i,-i}(\sigma\rho(a))}Y_{-i,j}(b-\sigma\rho(b))\cdot
\end{aligned}\\
\cdot\,^{X_{i,-i}(\sigma\rho(a))X_{-i,j}(\sigma\rho(b))}Y_{i,-i}(-a+\sigma\rho(a)),\\
X_{i,-i}(\rho(a))X_{i,-j}(\rho(b))X_{i,-i}(-\rho(a))\Big).
\end{multline*}
Rewriting 
\begin{multline*}
\,^{X_{-i,j}(\sigma\rho(b))}Y_{i,-i}(-a+\sigma\rho(a))=\\
=Y_{i,-i}(-a+\sigma\rho(a))[Y_{i,-i}(a-\sigma\rho(a)),\,X_{-i,j}(\sigma\rho(b))\rrbracket
\end{multline*}
we get with the use of KL7
\begin{multline*}
\psi\Big(X_{i,-i}(a)X_{-i,j}(b)X_{i,-i}(-a)\Big)=\\
=\Big(\,^{X_{i,-i}(a)}Y_{-i,j}(b-\sigma\rho(b))[Y_{i,-i}(a-\sigma\rho(a)),\,X_{-i,j}(\sigma\rho(b))\rrbracket,\,\\
X_{i,-i}(\rho(a))X_{i,-j}(\rho(b))X_{i,-i}(-\rho(a))\Big),
\end{multline*}
and finally
\begin{multline*}
\begin{aligned}&\psi\Big(X_{i,-i}(a)X_{-i,j}(b)X_{i,-i}(-a)\Big)=\\&=\Big(Y_{ij}((ab-\sigma\rho(ab))\cdot\sign i)\cdot\end{aligned}\\
\cdot Y_{-i,j}(b-\sigma\rho(b))Y_{-j,j}(-ab^2+2\,\sigma\rho(ab)b-\sigma\rho(ab^2)),\\
X_{ij}(\sigma\rho(ab)\cdot\sign i)X_{-j,j}(-\sigma\rho(ab^2))X_{-i,j}(\sigma\rho(b))\Big).
\end{multline*}
On the other hand,
\begin{multline*}
\begin{aligned}&\psi\Big(X_{ij}(ab\cdot\sign i)X_{-j,j}(-ab^2)X_{-i,j}(b)\Big)=\\&=\Big(Y_{ij}((ab-\sigma\rho(ab))\cdot\sign i)\cdot\end{aligned}\\
\cdot Y_{-j,j}(-ab^2+\sigma\rho(ab^2))\cdot\,^{X_{ij}(\sigma\rho(ab)\sign i)}Y_{-i,j}(b-\sigma\rho(b)),\\
X_{ij}(\sigma\rho(ab)\cdot\sign i)X_{-j,j}(-\sigma\rho(ab^2))X_{-i,j}(\sigma\rho(b))\Big).
\end{multline*}
Other relations are similar and much less tedious.

Obviously, $\iota\leftthreetimes\sigma^*\circ\psi=1$ and it only remains to show that $\psi$ is surjective. All elements of types $(1,\,X_{ij}(s))$ and $(Y_{ij}(a),\,1)$ lie in the image of $\psi$, and then elements of type $(\,^{X_{hk}(r)}Y_{ij}(a),\,1)$ lie as well.
\end{proof}

In the proof of the local--global principle we also need another presentation for the relative Steinberg group. It is inspired by the definition of the relative linear Steinberg groups given by Tulenbaev.

\begin{df}
Let {\it Tulenbaev relative symplectic Steinberg group} $\St\!\Sp^{\rm T}_{2n}(R,\,I)$ be a group defined by the set of generators 
$$
\{[u,\,v,\,a,\,b]\in\Ep_{2n}(R)e_1\times R^{2n}\times I\times I\mid \lan u,\,v\ran=0\}
$$
subject to the relations
\setcounter{equation}{-1}
\renewcommand{\theequation}{T\arabic{equation}}
\begin{align}
&[u,\,vr,\,a,\,b]=[u,\,v,\,ra,\,b]\ \ \forall\,r\in R,\\
&[u,\,v,\,a,\,b][u,\,w,\,a,\,c]=[u,\,v,\,a,\,b+c+a^2\lan v,\,w\ran],\\
&[u,\,v,\,a,\,0][u,\,v,\,b,\,0]=[u,\,v,\,a+b,\,0],\\
&[u,\,u,\,a,\,0]=[u,\,0,\,0,\,2a],\\
&[u,\,v,\,a,\,0]=[v,\,u,\,a,\,0]\ \ \forall\,(u,\,v)\in\Ep_{2n}(R)(e_1,\,e_2),\\
&\begin{aligned}\!
[u+vr,\,0,\,0,\,a]=[u,\,0,\,0,\,a][v,\,&0,\,0,\,ar^2][u,\,v,\,ar,\,0]\\&\forall\,r\in R\ \ \forall\,(u,\,v)\in\Ep_{2n}(R)(e_1,\,e_2),
\end{aligned}\\
&\begin{aligned}\!
[u',\,v',\,a',\,b'][u,\,v,\,a,\,b][u',\,&v',\,a',\,b']\inv=\\
&=[T(u',\,v'a',\,b')u,\,T(u',\,v'a',\,b')v,\,a,\,b].
\end{aligned}
\end{align}
\end{df}

There is a natural map $\kappa\colon\St\!\Sp^{\rm T}_{2n}(R,\,I)\rightarrow\St\!\Sp_{2n}(R)$ sending $[u,\,v,\,a,\,b]$ to $[u,\,va,\,b]$ (here we need the relation K7 established in the previous section). Its image is contained in $\Ker\big(\St\!\Sp_{2n}(R)\epi\St\!\Sp_{2n}(R/I)\big)$ and contains all elements of the form $\,^{g}X_{ij}(a)=[\phi(g)e_i,\,\phi(g)e_{-j}\,a\,\sign{-j},\,0]$ and $\,^{g}X_{i,-i}(a)=[\phi(g)e_i,\,0,\,a]$ for $a\in I$, and thus actually coincides with this kernel.

Any triple $(u,\,v,\,a)\in V\times V\times R$ defines a homomorphism $$\alpha_{u,v,a}\colon\St\!\Sp^{\rm T}_{2n}(R,\,I)\rightarrow\St\!\Sp^{\rm T}_{2n}(R,\,I)$$ sending generator $[u',\,v',\,a',\,b']$ to $[T(u,\,v,\,a)u',\,T(u,\,v,\,a)v',\,a',\,b']$. To show that $\alpha_{u,v,a}$ is well-defined we have to check that T0--T6 hold for the images of the generators, but that is straightforward. Next, there exists a well-defined homomorphism $$\St\!\Sp_{2l}(R)\rightarrow\mathrm{Aut}\,(\St\!\Sp^*_{2l}(R,\,I))$$ sending $X(u,\,v,\,a)$ to $\alpha_{u,v,a}$, i.e., the absolute Steinberg group acts on Tulenbaev group. Obviously, we need to verify that K1--K3 hold for $\alpha_{u,v,a}$, but that is also straightforward.

\begin{lm}
\label{allthesame}
Let $I\trianglelefteq R$ be a splitting ideal. Then $$\St\!\Sp^{\rm T}_{2l}(R,\,I)=\St\!\Sp_{2l}(R,\,I)=\St\!\Sp^{\rm KL}_{2l}(R,\,I).$$
\end{lm}
\begin{proof}

We identify $\St\!\Sp_{2l}(R,\,I)$ with $\St\!\Sp^{\rm KL}_{2l}(R,\,I)$ and construct a map inverse to $\kappa$. With this end define 
$$
Y^*_{ij}(a)=[e_i,\,e_{-j},\,a\,\sign{-j},\,0]\quad\text{and}\quad Y^*_{i,-i}(a)=[e_i,\,0,\,0,\,a]
$$
inside $\St\!\Sp^{\rm T}_{2l}(R,\,I)$. These elements satisfy relations KL0--KL7. KL0--KL2 and KL7 are obvious. Consider, say, KL4.
\begin{multline*}
\llbracket X_{i,-i}(r),\,Y^*_{-i,j}(b)]=\\
=[e_{-j},\,T(e_i,\,0,\,r)e_{-i},\,b\eps{-j},\,0][e_{-j},\,e_{-i},\,b\eps{-j},\,0]\inv=\\
=[e_{-j},\,e_ir\eps i,\,b\eps{-j},\,-rb^2]=\\
=[e_{-j},\,e_ir\eps i,\,b\eps{-j},\,0][e_{-j},\,0,\,b\eps{-j},\,-rb^2]=\\
=Y^*_{ij}(rb\eps i)\cdot Y^*_{-j,j}(-rb^2);
\end{multline*}
One can check other relations similarly. For KL5 use T5 and for KL6 use T3. Thus, we have a map $\theta\colon\St\!\Sp^{\rm KL}_{2l}(R,\,I)\rightarrow\St\!\Sp^{\rm T}_{2l}(R,\,I)$ preserving the action. Obviously, $\kappa\theta=1$ thus $\theta$ is injective. It remains to show that it is also surjective. First, observe that $[e_1,\,v,\,a,\,b]=[e_1,\,0,\,0,\,b-a^2\sum v_kv_{-k}]\prod[e_1,\,e_k,\,v_ka,\,0]$ lie in the image of $\theta$ (here $v_k$ is a $k$-th coordinate of $v$; we use that $v_{-1}=0$ and T3). Thus, all generators $[u,\,v,\,a,\,b]$ lie in the image of $\theta$, since it preserves the action.
\end{proof}

In the sequel for splitting ideals we do not distinguish relative Steinberg groups defined in this section.

\section{Local-global principle}

In this section we prove the Main Theorem.

Fix a non-nilpotent element $a\in R$. Let $\lambda_a\colon R\rightarrow R_a$ be a principal localisation of $R$ in $a$.

For any $x\in R[t]$ consider the evaluation map $\mathrm{ev}_x\colon R[t]\rightarrow R[t]$, which is the only $R$-algebra homomorphism sending $t$ to $x$. For $p\in R[t]$ denote its image under $\mathrm{ev}_x$ by $p(x)$, e.g., $p=p(t)$. Similarly, for $g\in\St\!\Sp_{2n}(R[t])$ denote its image under $\mathrm{ev}_{x}^*$ by $g(x)$. We claim the following.

\begin{lm}
\label{tul-inj}
Consider $g(t)\in\St\!\Sp_{2n}(R[t],\,tR[t])$ such that $$\lambda_{a}^*(g(t))=1\in\St\!\Sp_{2n}(R_a[t]).$$ Then there exists an $N\in\mathbb N$ such that $g(a^Nt)=1$. Similarly, assume that
$$\lambda_{a}^*(g(t))\in\mathrm{Im}\big(\St\!\Sp_{2n-2}(R_a[t])\rightarrow\St\!\Sp_{2n}(R_a[t])\big).$$
Then there exists an $N\in\mathbb N$ such that 
$$g(a^Nt)\in\mathrm{Im}\big(\St\!\Sp_{2n-2}(R[t],\,tR[t])\rightarrow\St\!\Sp_{2n}(R[t],\,tR[t])\big).$$
\end{lm}

Now, we define the symplectic analogue of Tulenbaev map and use it to prove Lemma~\ref{tul-inj}

\begin{df}
Denote
$
B=R\ltimes tR_a[t]
$
the ring with component-wise addition and multiplication given by
$$
(r,\,f)\cdot(s,\,g)=(rs,\,\lambda_a(r)g+f\lambda_a(s)+fg).
$$
One may think of elements of $B$ as polynomials in $t$ with the constant term from $R$ and all other coefficients from $R_a$. 
\end{df}

Consider a direct system of rings
$$
\xymatrix{
R[t]\ar@<-0.0ex>[r]^{\mathrm{ev}_{at}}&R[t]\ar@<-0.0ex>[r]^{\mathrm{ev}_{at}}&R[t]\ar@<-0.0ex>[r]^{\mathrm{ev}_{at}}&\cdots
}
$$
i.e., $(S_i,\,\psi_{ij})_{0\leq i\leq j}$, where $S_i=R[t]$ and $\psi_{ij}\colon t\mapsto a^{j-i}t$. It induces a direct system of Steinberg groups. The following facts are left to the reader.

\begin{lm}  
A system of maps $\varphi_i\colon S_i\rightarrow B$ sending $$p(t)\mapsto\big(p+tR[t],\,\lambda_{a}^*(p)(a^{-i}t)-\lambda_{a}^*(p)(0)\big)$$ induces

\begin{enumerate}
\item
 an isomorphism
$$
\varinjlim S_i\rightarrow ^{\!\!\!\!\!\!\!\sim}B;
$$
\item
an isomorphism
$$
\varinjlim\St\!\Sp_{2n}(S_i)\rightarrow ^{\!\!\!\!\!\!\!\sim}\ \St\!\Sp_{2n}(B).
$$
Indeed, Steinberg group functor commutes with directed limits.
\end{enumerate}
\end{lm}

Now, we claim that the composition of $\varphi_{0}^*$ with the inclusion
$$
\xymatrix{
\mu\colon\St\!\Sp_{2n}(R[t],\,tR[t])\ar@{^(->}[r]^{}&\St\!\Sp_{2n}(R[t])\ar@<-0.0ex>[r]^{\varphi_{0}^*}&\St\!\Sp_{2n}(B)
}
$$
factors through the localisation in $a$. More generally, the following statement holds.
\begin{lm}[Tulenbaev map]
\label{tul}
Let $B$ be a ring, $a\in B$, and $I\trianglelefteq B$ be an ideal such that for any $x\in I$ there exists a unique $y\in I$ such that $ya=x$ {\rm(}equivalently, a localisation map $\lambda_a\colon I\rightarrow I_a=I\otimes_RR_a$ is an isomorphism{\rm)}. Then, there exists a map
$$
\mathrm T\colon\St\!\Sp_{2n}^{\mathrm T}(B_a,\,I_a)\rightarrow\St\!\Sp_{2n}(B)
$$
making the diagram
$$
\xymatrix{
\St\!\Sp_{2n}^{\mathrm T}(B,\,I)\ar@<-0.0ex>[rr]^{}\ar@<-0.0ex>[d]_{\lambda_{a}^*}&&\St\!\Sp_{2n}(B)\ar@<-0.0ex>[d]_{\lambda_{a}^*}\\
\St\!\Sp_{2n}^{\mathrm T}(B_a,\,I_a)\ar@<-0.0ex>[rr]^{}\ar@{-->}[rru]_{\mathrm T}&&\St\!\Sp_{2n}(B_a)\\
}
$$
commutative. Moreover, for $g\in\mathrm{Im}\big(\St\!\Sp_{2n-2}^{\mathrm T}(R_a[t],tR_a[t])\rightarrow\St\!\Sp_{2n}^{\mathrm T}(B_a,\,I_a)\big)$ one has $\mathrm T(g)\in\mathrm{Im}\big(\St\!\Sp_{2n-2}(B)\rightarrow\St\!\Sp_{2n}(B)\big)$.
\end{lm}

The next section is devoted to the proof of Lemma~\ref{tul}. Now, we deduce Lemma~\ref{tul-inj} from it.
\begin{proof}[Proof of Lemma~\ref{tul-inj}]
Apply Lemma~\ref{tul} for $a\in R\subseteq B$, $B=R\ltimes tR_a[t]$ as above, $I=tR_a[t]\trianglelefteq B$. Consider the following commutative diagram.
$$
\xymatrix{
\St\!\Sp_{2n}(R[t],\,tR[t])\ar@{^(->}[rr]^{}\ar@<-0.0ex>[rd]_{\varphi_0^*}\ar@<-0.0ex>[dd]_{\lambda_{a}^*}&&\St\!\Sp_{2n}(R[t])\ar@<-0.0ex>[dd]_{\varphi_{0}^*}\\
&\St\!\Sp_{2n}(B,\,I)\ar@{^(->}[rd]^{}\ar@<-0.0ex>[ld]_{\lambda_a^*}&\\
\St\!\Sp_{2n}(R_a[t],\,tR_a[t])\ar@<-0.0ex>[rr]^{\mathrm T}&&\St\!\Sp_{2n}(B)\\
}
$$

Take $g(t)\in\St\!\Sp_{2n}(R[t],\,tR[t])$ such that $\lambda_{a}^*(g(t))=1$. Then 
$$\varphi_0^*(g(t))=(T\circ\lambda_{a}^*)(g(t))=1$$ 
as well, i.\,e., $g(t)$ becomes trivial in the limit. But it can only happen if $\psi_{0,N}^*(g(t))=1$ for some $N$ (use the construction of direct limit as disjoint union modulo an equivalence relation). The proof of the second statement is similar.
\end{proof}

For the next lemma the proof of Lemma~16 of~\cite{n8} works verbatim. There are two references in that proof: instead of Lemma~8 of~\cite{n8} use Lemma~\ref{allthesame}, and instead of Lemma~15 of~\cite{n8} use Lemma~\ref{tul-inj}. The second statement is not proven in~\cite{n8}, but the proofs of both statements are the same.

\begin{lm}
\label{almost}
Consider $a$, $b\in R$ generating $R$ as an ideal, $Ra+Rb=R$. Assume that for $g\in\St\!\Sp_{2n}(R[t], tR[t])$ one has $\lambda_{a}^*(g)=\lambda_{b}^*(g)=1$ . Then $g=1$. Similarly, assume that $\lambda_{a}^*(g)\in\St\!\Sp_{2n-2}(R_a[t], tR_a[t])$ and $\lambda_{b}^*(g)\in\St\!\Sp_{2n-2}(R_b[t], tR_b[t])$. Then $g\in\St\!\Sp_{2n-2}(R[t], tR[t])$.
\end{lm}

Now, the Main Theorem also follows. For the first statement the proof of Theorem~2 of~\cite{n8} can be repeated verbatim. The only reference in that proof is Lemma~16 of~\cite{n8}, which should be replaced by Lemma~\ref{almost}. For the second statement the same proof works.

\section{Tulenbaev map}

This section is devoted to the construction of the map
$$
\mathrm T\colon\St\!\Sp_{2n}^{\mathrm T}(B_a,\,I_a)\rightarrow\St\!\Sp_{2n}(B) 
$$
from Lemma~\ref{tul}. As there, let $B$ be a ring, $a\in B$ a non-nilpotent element, $I\trianglelefteq B$, such that for any $x\in I$ there exists a unique $y\in I$ such that $ya=x$. We denote such a $y$ by $\frac xa$. Elements $\frac x{a^N}$ are also well-defined. The localisation map $\lambda_a\colon I\rightarrow I_a$ is an isomorphism and we identify $I$ and $I_a$.

To define the map $\mathrm T$ we need to find elements $Z(u,\,v,\,b,\,c)\in\St\!\Sp_{2n}(B)$ for any $u\in\Ep_{2n}(B_a)e_1$, $v\in B_a^{2n}$, $\lan u,\,v\ran=0$, and $b$, $c\in I$ subject to relations T0--T6. We start with the following definition.

\begin{df}
For $u$, $v_1,\ldots,v_N\in B^{2n}$ such that $\lan u,\,v_k\ran=0$ for all $k$ define
\begin{multline*}
Z(u;\,v_1,\ldots,v_N)=X(u,\,v_1,\,0)\ldots X(u,\,v_N,\,0)\cdot X(u,\,0,\,-\sum\limits_{i<j}\lan v_i,\,v_j\ran).
\end{multline*}
\end{df}

\begin{lm}
Consider $u$, $v$, $w\in B^{2n}$, $\lan u,\,v\ran=\lan u,\,w\ran=0$. Suppose that $w$ has a pair of zero coordinates, i.e., $w_i=w_{-i}=0$ for some $i$. Then 
$$
[X(u,\,v,\,0),\,X(u,\,w,\,0)]=X(u,\,0,\,2\lan v,\,w\ran).
$$
\end{lm}
\begin{proof}
Use Lemma~\ref{xlist}\,(X1) to compute the conjugate and decompose the result by Lemma~\ref{xlist}\,(X6)
\begin{multline*}
\,^{X(u,\,v,\,0)}X(u,\,w,\,0)=\\=X(u,\,0,\,-1)X(w+u\lan v,\,w\ran,\,0,\,-1)X(w+u(1+\lan v,\,w\ran),\,0,\,1),
\end{multline*}
then decompose the first and the third factors by Lemma~\ref{xlist2}\,(X8). Next, change the order of factors and then simplify the product using Lemmas~\ref{xlist} and \ref{xlist2}. As a result, we have
$$
\,^{X(u,\,v,\,0)}X(u,\,w,\,0)=X(u,\,0,\,2\lan v,\,w\ran)X(w,\,u,\,0).
$$
\end{proof}

For the following corollary use that the symmetric group is generated by fundamental transpositions.

\begin{cl}
Take $u$ and $v_1,\ldots,v_N\in B^{2n}$ such that $\lan u,\,v_k\ran=0$ and each $v_k$ has a pair of zero coordinates. Then for any permutation $\sigma\in S_N$ we have
$$
Z(u;\,v_1,\ldots,v_N)=Z(u;\,v_{\sigma(1)},\ldots,v_{\sigma(N)}).
$$
\end{cl}

\begin{df}
For $u$, $v_1,\ldots,v_N\in B^{2n}$ such that $\lan u,\,v_k\ran=0$ and each $v_k$ has a symmetric pair of zero coordinates we denote
$$
Z(u;\,\{v_k\}_{1\leq k\leq N})=Z(u;\,v_1,\ldots,v_N).
$$
\end{df}

Observe also that the following easy fact holds (use Lemma~\ref{xlist}\,(X2) and Lemma~\ref{xlist2}\,(X9).

\begin{lm}
\label{forgotten}
For $u$, $v_1,\ldots,v_N\in B^{2n}$ such that $\lan u,\,v_k\ran=0$ and each $v_k$ has a pair of zero coordinates, $r\in B$, one has
$$
Z(ur;\,\{v_k\}_{1\leq k\leq N})=Z(u;\,\{rv_k\}_{1\leq k\leq N}).
$$
\end{lm}

The following result is well-known (see~\cite{n2,n5}).

\begin{lm}[Symplectic Suslin's Lemma]
\label{suslin}
For $w$, $u$, $v\in B^{2n}$ such that $\lan w,\,u\ran=A\in B$, $\lan u,\,v\ran=0$ denote
$$
v_{ij}=v_{ij}^w=(e_iu_{-j}\,\sign j-e_ju_{-i}\,\sign i)(v_iw_j-v_jw_i)
$$
for any distinct $\,-n\leq i,j\leq n$. Then one has $v_{ij}=v_{ji}$, $\lan u,\,v_{ij}\ran=0$, and
$$
\sum_{i<j}v_{ij}=vA.
$$
\end{lm}

Compare the next result with Lemma~\ref{xlist2}\,(X10): we do not need that $v_i=v_{-i}=0$, but we assume that $\lan v,\,w\ran=0$.

\begin{lm}
\label{orth}
Take $u$, $v$, $w\in B^{2n}$ such that $\lan u,\,v\ran=\lan u,\,w\ran=\lan v,\,w\ran=0$ and $w_i=w_{-i}=0$. Then
$$
X(u,\,v+w,\,0)=X(u,\,v,\,0)X(u,\,w,\,0).
$$
\end{lm}
\begin{proof}
By Lemma~\ref{xlist}\,(X6) we have
$$
X(u,\,v+w,\,0)=X(u,\,0,\,-1)X(v+w,\,0,\,-1)X(u+v+w,\,0,\,1).
$$
Decompose the second and the third factors by Lemma~\ref{xlist2}\,(X8). We get a factor $X(w,\,u+v,\,0)$, and decompose it by Lemma~\ref{xlist2}\,(X10). Vectors $u$, $v$, $w$ are orthogonal, thus all factors commute. Simplify the product by Lemma~\ref{xlist2}\,(X10) and get the claim with the use of Lemma~\ref{xlist}\,(X6).
\end{proof}

\begin{lm}
\label{x=z}
Take $w$, $u$, $v\in B^{2n}$ such that $\lan w,\,u\ran=A$ and $\lan u,\,v\ran=0$. Assume in addition that $v$ has a symmetric pair of zero coordinates. Then
$$
X(u,\,vA,\,0)=Z(u,\,\{v_{ij}\}_{i\leq j}).
$$
\end{lm}
\begin{proof}
Say, $v_1=v_{-1}=0$. Then $v_{1,-1}=0$. Decompose
$$
vA=\sum_{i<j}v_{ij}=\underbrace{\sum_{i\neq\pm1}v_{-1,i}}_p+\underbrace{\sum_{i\neq\pm1}v_{1,i}}_{q}+\underbrace{\sum_{i,j\neq\pm1}v_{ij}}_r
$$
with the use of Suslin's Lemma (Lemma~\ref{suslin}). Obviously,
$$
p_{-1}=\Big(\sum_{i\neq\pm1}v_{-1,i}\Big)_{-1}=\Big(\sum_{i<j}v_{ij}\Big)_{-1}=v_{-1}A=0
$$
and similarly $q_1=v_1A=0$. Also, $p_1=q_{-1}=r_{-1}=r_1=0$. Thus, by Lemma~\ref{xlist2}\,(X10) we get
\begin{multline*}
X(u,\,vA,\,0)=X(u,\,p+q+r,\,0)=\\
=X(u,\,p,\,0)X(u,\,q,\,0)X(u,\,r,\,0)X(u,\,0,\,-\lan p,\,q\ran-\lan p,\,r\ran-\lan q,\,r\ran).
\end{multline*}
For $1<i\leq n$ denote $z_i=v_{-1,i}+v_{-1,-i}$. Then $\lan z_i,\,z_j\ran=0$ for $i\neq j$, each $z_i$ has a pair of zero coordinates and $\sum_{i=2}^nz_i=p$. By Lemma~\ref{orth} we get
$$
X(u,\,p,\,0)=\prod_{i=2}^nX(u,\,z_i,\,0).
$$
Next, observe that $v_{-1,i}$ and $v_{-1,-i}$ have a common symmetric pair of zero coordinates. Thus, by Lemma~\ref{xlist2}\,(X10), one has
$$
X(u,\,z_i,\,0)=X(u,\,v_{-1,i},\,0)X(u,\,v_{1,i},\,0)X(u,\,0,\,-\lan v_{-1,i},\,v_{1,i}\ran),
$$
so that
$$
X(u,\,p,\,0)=Z(u;\,\{v_{-1,i}\}_{i\neq\pm1}).
$$
Similarly, $X(u,\,q,\,0)=Z(u;\,\{v_{1,i}\}_{i\neq\pm1})$ and by Lemma~\ref{xlist2}\,(X10), $X(u,\,r,\,0)=Z(u;\,\{v_{ij}\}_{i,j\neq\pm1})$ what finishes the proof.
\end{proof}

\begin{lm}
\label{decomposition}
Take $w$, $u$, $v\in B^{2n}$, $\lan w,\,u\ran=A$ and $\lan u,\,v\ran=0$. Consider $x^1,\ldots x^N\in B^{2n}$ such that each $x^k$ has a pair of zero coordinates, $\lan u,\,x^k\ran=0$, and $\sum_{k=1}^Nx^k=vA$. Then one has
$$
Z(u;\,\{x^kA\}_{k=1}^N)=Z(u;\,\{v_{ij}A\}_{i<j}).
$$
\end{lm}
\begin{proof}
Since $\lan u,\,x^k\ran=0$ consider $x_{ij}^k=(x^k)_{ij}^w$ from Suslin's Lemma (Lemma~\ref{suslin}) and use Lemma~\ref{x=z} to get
$$
X(u,\,x^kA,\,0)=Z(u,\,\{x_{ij}^k\}_{i<j}).
$$
Then, 
$$
Z(u;\,\{x^kA\}_{k=1}^N)=Z(u,\,\{x_{ij}^k\}_{k,\,i<j}).
$$
On the other hand, for fixed $i$ and $j$ all $x_{ij}^k$ are scalar multiples of the same vector having a pair of zero coordinates and
$$
\sum_{k=1}^Nx_{ij}^k=(e_iu_{-j}\,\sign j-e_ju_{-i}\,\sign i)\Big(\big(\sum_{k=1}^Nx_i^k\big)w_j-\big(\sum_{k=1}^Nx_j^k\big)w_i\Big)=v_{ij}A.
$$
Thus,
$$
X(u,\,v_{ij}A,\,0)=\prod_{k=1}^NX(u,\,x_{ij}^k,\,0)
$$
by Lemma~\ref{xlist2}\,(X10) and
$$
Z(u;\,\{v_{ij}A\}_{i<j})=Z(u,\,\{x_{ij}^k\}_{k,\,i<j}).
$$
\end{proof}

\begin{df}
Take $u$, $v\in B^{2n}$ such that $\lan u,\,v\ran=0$ and denote by
$$I(u)=\sum_{k=-n}^nBu_k$$ 
the ideal generated by entries of $u$. Then for an $A\in I(u)$ take any $w\in B^{2n}$ such that $\lan w,\,u\ran=A$ and denote
$$
Z^A(u,\,v)=Z(u;\,\{v_{ij}^wA\}_{i<j}).
$$
By the previous lemma, this element does not depend on the choice of $w$. The projection of $Z^A(u,\,v)$ to the elementary group 
$$
\phi(Z^A(u,\,v))=T(u,\,vA^2,\,0).
$$
\end{df}

We start to prove properties of the elements $Z^A(u,\,v)$.

\begin{lm}
\label{conj}
Take $u$, $v\in B^{2n}$ such that $\lan u,\,v\ran=0$, $A\in I(u)$ and $g\in\St\!\Sp_{2n}(B)$. Then
$$
g\,Z^A(u,\,v)g\inv=Z^A(\phi(g)u,\,\phi(g)v).
$$
\end{lm}

\begin{rk}
Take $w$ such that $\lan w,\,u\ran=A$. Then, $\lan \phi(g)w,\,\phi(g)u\ran=\lan w,\,u\ran=A$ as well, so that $A\in I(\phi(g)u)$ and the right hand side is well-defined.
\end{rk}
\begin{proof}
We may assume that $g=X_{ij}(b)$. Obviously,
$$
g\,X(u,\,v_{hk}A,\,0)g\inv=X(\phi(g)u,\,\phi(g)v_{hk}A,\,0)
$$
and
$$
-\sum\lan v_{hk}A,\,v_{st}A\ran=-\sum\lan\phi(g)v_{hk}A,\,\phi(g)v_{st}A\ran,
$$
so that,
$$
g\,Z^A(u,\,v)g\inv=Z(\phi(g)u;\,\{\phi(g)v_{hk}A\}_{h<k}).
$$
For $n\geq4$ each $T_{ij}(b)v_{hk}$ still has at least one pair of zero coordinates what finishes the proof for this case by Lemma~\ref{decomposition}.

Now, consider the case of $n=3$. For $j=-i$ any $T_{ij}(b)v_{hk}$ still has a pair of zero coordinates. Now, assume $j\neq\pm i$. If $h$, $k\not\in\{j,-i\}$, then $\phi(g)v_{hk}=v_{hk}$. If $\{h,k\}=\{j,-i\}$ we also get that $T_{ij}(b)v_{hk}$ has a pair of zero coordinates. Thus, we may assume that $h\in\{j,-i\}$, say, $h=-i$, and $k\not\in\{\pm i,\pm j\}$.

Set 
$$u_{k,-i}=e_ku_{i}\,\sign{-i}-e_{-i}u_{-k}\,\sign k,$$ 
then $v_{k,-i}=u_{k,-i}(v_kw_{-i}-v_{-i}w_k)$.
One has
\begin{align*}
&T_{ij}(b)u=u+e_iu_jb-e_{-j}u_{-i}\,b\,\sign{ij},\\
&T_{ij}(b)u_{k,-i}=u_{k,-i}+e_{-j}u_{-k}\,b\,\sign{ijk}.
\end{align*}
Direct computation shows that
$$
\lan T_{ij}(b)u,\,u_{k,-i}-e_ku_j\,b\,\sign i\ran=0.
$$
Set 
$$q=u_{k,-i}-e_ku_j\,b\,\sign i\,\text{ and }\,r=e_ku_j\,b\,\sign i+e_{-j}u_{-k}\,b\,\sign{ijk}.$$ 
One has $T_{ij}(b)u_{k,-i}=q+r$, so that $r$ is also orthogonal to $T_{ij}(b)u$. Both $q$ and $r$ have a pair of zero coordinates and, moreover, they are orthogonal. Set $c=(v_kw_{-i}-v_{-i}w_k)$, then by Lemma~\ref{orth}
$$
X(T_{ij}(b)u,\,T_{ij}(b)v_{k,-i}A,\,0)=X(T_{ij}(b)u,\,qcA,\,0)X(T_{ij}(b)u,\,rcA,\,0).
$$
Finally, the claim follows from Lemma~\ref{decomposition}.
\end{proof}

The next lemma follows from Lemma~\ref{decomposition}.

\begin{lm}
\label{add}
Take $u$, $v$, $w\in B^{2n}$ such that $\lan u,\,v\ran=\lan u,\,w\ran=0$, $A\in I(u)$. Then
$$
Z^A(u,\,v)Z^A(u,\,w)=Z^A(u,\,v+w)X(u,\,0,\,\lan v,\,w\ran\cdot A^4).
$$
\end{lm}

\begin{cl}
For $u$, $v$, $\lan u,\,v\ran=0$, $A\in I(u)$ one has
\begin{align*}
&Z^A(u,\,0)=1,\\
&Z^A(u,\,v)\inv=Z^A(u,\,-v).
\end{align*}
\end{cl}

\begin{lm}
\label{symm}
Take $u$, $v\in B^{2n}$ such that $\lan u,\,v\ran=0$, $A\in I(u)\cap I(v)$, $b\in B$. Assume that there exist $p$, $q\in B^{2n}$ such that
$$\lan u,\,p\ran=\lan u,\,q\ran=\lan v,\,p\ran=\lan v,\,q\ran=0,$$ 
and $\lan p,\,q\ran=A$. Then one has
$$
Z^A(u,\,vb\cdot A^3)=Z^A(v,\,ub\cdot A^3).
$$
\end{lm}
\begin{proof}
Denote $g=Z^A(u,\,pb)$ and $h=Z^A(v,\,q)$. Compute the commutator in two ways
$$\!\,^gh\cdot h\inv=g\cdot\,^hg\inv.$$
Recall that $\phi\big(Z^A(u,\,pb)\big)=T(u,\,pbA^2,\,0)$ and use Lemmas~\ref{conj} and \ref{add} to get
$$
\!\,^gh\cdot h\inv=Z^A(v,\,q+ubA^3)Z^A(v,\,-q)=Z^A(v,\,ubA^3).
$$
Similarly, $g\cdot\,^hg\inv=Z^A(u,\,vbA^3)$.
\end{proof}

\begin{lm}
Consider $w$, $u\in B^{2n}$, $b\in B$, denote $A=\lan w,\,u\ran$. Assume that there exist $z$, $v\in B^{2n}$ such that $\lan z,\,v\ran=A$ and
$$
\lan u,\,v\ran=\lan u,\,z\ran=\lan w,\,v\ran=\lan w,\,z\ran=0.
$$
Then one has
$$
Z^A(u,\,ubA^3)=X(u,\,0,\,2bA^5).
$$
\end{lm}
\begin{proof}
Set $g=Z^A(u,\,zb)$, $h=Z^A(u,\,v)$ and compute $[g,\,h]$ in two ways. On the one hand,
$$
\!\,^gh\cdot h\inv=Z^A(u,\,v+ubA^3)Z^A(u,\,-v)=Z(u,\,ubA^3)
$$
by Lemma~\ref{conj}. On the other hand, 
\begin{multline*}
gh\cdot g\inv h\inv=\\=Z^A(u,\,zb+v)X(u,\,0,\,bA^5)Z^A(u,\,-zb-v)X(u,\,0,\,bA^5)=\\=X(u,\,0,\,2bA^5)
\end{multline*}
by Lemma~\ref{add}.
\end{proof}

\begin{lm}
\label{5+6}
Consider $u$, $v\in B^{2n}$, $A\in I(u)\cap I(v)$, $b$, $c\in B$. Assume that there exist $w$, $z$, $x$, $y\in B^{2n}$, such that
$$
\lan w,\,u\ran=\lan z,\,v\ran=\lan x,\,y\ran=A
$$
and pairs $(w,\,u)$, $(z,\,v)$ and $(x,\,y)$ are mutually orthogonal. Then one has
$$
X(u+vb,\,0,\,cA^{11})=X(u,\,0,\,cA^{11})X(v,\,0,\,b^2cA^{11})Z^A(u,\,vbcA^9).
$$
\end{lm}
\begin{proof}
First, use Lemma~\ref{add}
\begin{multline*}
Z^A(u+vb,\,xA^3)Z^A(u+vb,\,ycA^3)=\\=Z^A(u+vb,\,(x+yc)A^3)X(u+vb,\,0,\,cA^{11}).
\end{multline*}
We want to show that
$$
Z^A(u+vb,\,xA^3)=Z^A(x,\,(u+vb)A^3).
$$
With this end, use Lemma~\ref{symm} with $p=z-wb$ and $q=v$. Next, decompose
$$
Z^A(x,\,(u+vb)A^3)=Z^A(x,\,uA^3)Z^A(x,\,vbA^3)
$$
by Lemma~\ref{add} and use Lemma~\ref{symm} with $p=z$ and $q=v$ to show that
$$
Z^A(x,\,uA^3)=Z^A(u,\,xA^3)
$$
and with $p=w$, $q=u$ to get that
$$
Z^A(x,\,vbA^3)=Z^A(v,\,xbA^3).
$$
Similarly, one shows that
$$
Z^A(u+vb,\,ycA^3)=Z^A(u,\,ycA^3)Z^A(v,\,ybcA^3)
$$
and
$$
Z^A(u+vb,\,-(x+yc)A^3)=Z^A(u,\,-(x+yc)A^3)Z^A(v,\,-(x+yc)bA^3).
$$
Decompose also 
\begin{align*}
&Z^A(u,\,-(x+yc)A^3)=Z^A(u,\,-ycA^3)Z^A(u,\,-xA^3)X(u,\,0,\,cA^{11}),\\
&Z^A(v,\,-(x+yc)bA^3)=Z^A(v,\,-ybcA^3)Z^A(v,\,-xbA^3)X(v,\,0,\,b^2cA^{11})
\end{align*}
by Lemma~\ref{add}. Now, we can express $X(u+vb,\,0,\,cA^{11})$ in terms of these ten elements. Most of the factors will cancel, but we will need to interchange positions of $Z^A(u,\,xA^3)$ and $Z^A(v,\,ybcA^3)$, thus we obtain their commutator as an extra factor
$$
[Z^A(u,\,xA^3),\,Z^A(v,\,ybcA^3)]=Z^A(u,\,vbcA^9).
$$
\end{proof}

Now, we focus on the case $A=a^N$.

\begin{df}
Take $b\in I$ and $u$, $v\in B^{2n}$, such that $a^N\in I(u)$ for some $N\in\mathbb N$, $\lan u,\,v\ran=0$.
Then set
$$
Z(u,\,v,\,b)=Z^{(a^N)}\Big(u,\,v\frac b{a^{2N}}\Big).
$$
This element does not depend on the choice of $N$. Take $w\in B^{2n}$ such that $\lan w,\,u\ran=a^N$, then $\lan wa^M,\,u\ran=a^{N+M}$ and by the very definition
$$
\Big(v\frac{b}{a^{2(N+M)}}\Big)_{ij}^{(wa^M)}=v_{ij}^{(wa^M)}\cdot\frac{b}{a^{2(N+M)}}=v_{ij}^w\cdot a^M\cdot\frac{b}{a^{2(N+M)}},
$$
so that
$$
Z^{(a^{N+M})}\Big(u,\,v\frac{b}{a^{2(N+M)}}\Big)=Z\Big(u;\,\Big\{\big(v_{ij}^w\frac{b}{a^{2N+M}}\big)a^{N+M}\Big\}_{i<j}\Big)=Z^{(a^N)}\Big(u,\,v\frac b{a^{2N}}\Big).
$$
Observe that $\phi\big(Z(u,\,v,\,b)\big)=T(u,\,vb,\,0)$.
\end{df}

Below we list the properties of our new elements $Z(u,\,v,\,b)$. They follow directly from the definition and Lemmas~\ref{conj}\,--\,\ref{5+6}.

\begin{lm}
\label{zlist}
Take $u$, $v$, $v'\in B^{2n}$ such that $a^N\in I(u)$ for some $N\in\mathbb N$, $\lan u,\,v\ran=\lan u,\,v'\ran=0$, $b$, $c\in I$, $r\in B$, $g\in\St\!\Sp_{2n}(B)$. Then one has
\setcounter{equation}{-1}
\renewcommand{\theequation}{Z\arabic{equation}}
\begin{align}
&\phi\big(Z(u,\,v,\,b)\big)=T(u,\,vb,\,0),\\
&Z(u,\,vr,\,b)=Z(u,\,v,\,rb),\\
&Z(u,\,v,\,b)Z(u,\,v',\,b)=Z(u,\,v+v',\,b)X(u,\,0,\,b^2\lan v,\,v'\ran),\\
&Z(u,\,v,\,b)Z(u,\,v,\,c)=Z(u,\,v,\,b+c),\\
&g\,Z(u,\,v,\,b)g\inv=Z(\phi(g)u,\,\phi(g)v,\,b).
\end{align}
Assume that there also exist $w$, $z\in B^{2n}$ such that holds
$\lan w,\,u\ran=\lan z,\,v\ran=a^N$ and pairs $(w,\,u)$, $(z,\,v)$ are orthogonal. Then one also has
\begin{align}
&Z(u,\,u,\,b)=X(u,\,0,\,2b).
\end{align}
If in addition there exist $x$, $y\in B^{2n}$ such that $\lan x,\,y\ran=a^N$ and the pair $(x,\,y)$ is orthogonal to pairs $(w,\,u)$ and $(z,\,v)$, then
\begin{align}
&Z(u,\,v,\,b)=Z(v,\,u,\,b),\\
&X(u+vr,\,0,\,b)=X(u,\,0,\,b)X(v,\,0,\,br^2)Z(u,\,v,\,br).
\end{align}
\end{lm}

We need yet another property of $Z(u,\,v,\,b)$.

\begin{lm}
\label{onemore}
For $u$, $v\in B^{2n}$, $b\in I$, $M$, $N\in\mathbb N$, such that $a^N\in I(u)$, $\lan u,\,v\ran=0$, holds
$$Z(ua^M,\,v,\,b)=Z(u,\,v,\,a^Mb).$$
\end{lm}
\begin{proof}
Firstly, we clearify the notations. Take $w\in B^{2n}$ such that $\lan w,\,u\ran=a^N$, then $\lan w,\,ua^M\ran=a^{N+M}$. Denote $u_{ij}=e_iu_{-j}\,\sign j-e_ju_{-i}\,\sign i$, then $(ua^M)_{ij}=u_{ij}a^M$. Denote as usually $v_{ij}=u_{ij}(v_iw_j-v_jw_i)$. Then in the definition of 
$$Z^{a^{N+M}}\Big(ua^M,\,v\frac b{a^{2N+2M}}\Big)$$ 
we actually use $(ua^M)_{ij}\cdot(v_iw_j-v_jw_i)=v_{ij}a^M$. Thus, we have
\begin{multline*}
Z(ua^M,\,v,\,b)=Z^{a^{N+M}}\Big(ua^M,\,v\frac b{a^{2N+2M}}\Big)=\\
=Z\Big(ua^M;\,\Big\{\big(v_{ij}a^M\frac{b}{a^{2N+2M}}\big)a^{N+M}\Big\}_{i<j}\Big).
\end{multline*}
Now use Lemma~\ref{forgotten} and get
$$
Z\Big(ua^M;\,\Big\{\big(v_{ij}\frac{b}{a^{2N}}\big)a^{N}\Big\}_{i<j}\Big)=Z\Big(u;\,\Big\{\big(v_{ij}\frac{a^Mb}{a^{2N}}\big)a^{N}\Big\}_{i<j}\Big)=Z(u,\,v,\,a^Mb).
$$
\end{proof}

Finally, introduce yet another notation.

\begin{df}
For $u$, $v\in B^{2n}$, such that $a^N\in I(u)$ for some $N\in\mathbb N$, $\lan u,\,v\ran=0$, $b$, $c\in I$ denote
$$
Z(u,\,v,\,b,\,c)=Z(u,\,v,\,b)X(u,\,0,\,c).
$$
\end{df}

At this point, we are ready to construct Tulenbaev map 
$$\mathrm T\colon\St\!\Sp_{2n}^{\mathrm T}(B_a,\,I)\rightarrow\St\!\Sp_{2n}(B).$$
\begin{proof}[Proof of Lemma~\ref{tul}]
For each quadruple 
$$(u,\,v,\,b,\,c)\in\Big(\Ep_{2n}(B_a)e_1\Big)\times B_a^{2n}\times I\times I$$
we associate an element in $\St\!\Sp_{2n}(B)$. We proceed as follows. First, if $u=Me_1$, denote $w=-Me_{-1}$, then $\lan w,\,u\ran=1$. Next, $w$, $u$, $v\in B_a$, thus there exists an $N\in\mathbb N$ such that the entries of $wa^N$, $ua^N$, $va^N$ do not have denominators, i.e., lie in the image of the localisation homomorphism $\lambda_a\colon B\rightarrow B_a$. Then for each of these elements take its their preimages and get vectors $\tilde w$, $\tilde u$, $\tilde v\in B^{2n}$. Since $\lan\tilde u,\,\tilde v\ran$ localises to zero and $\lan \tilde w,\,\tilde u\ran$ localises to $a^{2N}$, there exist $M\in\mathbb N$ such that $\lan\tilde u,\,\tilde va^M\ran=0$ and $\lan\tilde wa^M,\,\tilde u\ran=a^{2N+M}$. Then the element
$$
Z\Big(\tilde u,\,\tilde va^M,\,\frac b{a^{2N+M}},\,\frac c{a^{2N}}\Big)
$$
in $\St\!\Sp_{2n}(B)$ is defined. Using Lemma~\ref{xlist}(X2), Lemma~\ref{zlist}(Z1) and Lemma~\ref{onemore} one can show that this element does not depend on the above choices. Thus, we have a well-defined set-theoretic map from the set of generators of $\St\!\Sp_{2n}^{\mathrm T}(B_a,\,I)$ to $\St\!\Sp_{2n}(B)$. 

Next, we need to show that the images of $(u,\,v,\,b,\,c)$ under this map satisfy relations T0--T6. This is a straightforward consequence of Lemmas~\ref{xlist} and \ref{zlist} and the fact that the above map is well-defined. The only trick one should use to get T3--T5 is the following. For $u=Me_1$, where $M\in\Ep_{2n}(B_a)$, one can take $w=-Me_{-1}$, $v=Me_{-2}$, $z=Me_2$ and use their lifts to deduce T3 from Lemma~\ref{zlist}(Z5). Similarly, one can use $(e_3,\,e_{-3})$ for T4 and T5.

Now, we have to show that the diagram 
$$
\xymatrix{
\St\!\Sp_{2n}^{\mathrm T}(B,\,I)\ar@<-0.0ex>[rr]^{\kappa}\ar@<-0.0ex>[d]_{\lambda_{a}^*}&&\St\!\Sp_{2n}(B)\ar@<-0.0ex>[d]_{\lambda_{a}^*}\\
\St\!\Sp_{2n}^{\mathrm T}(B_a,\,I_a)\ar@<-0.0ex>[rr]_{\kappa}\ar@<-0.0ex>[rru]_{\mathrm T}&&\St\!\Sp_{2n}(B_a)\\
}
$$
is commutative. Start with the upper triangle. Lemma~\ref{zlist}\,(Z7) and Lemma~\ref{xlist2}\,(X8) imply that for $u$ with a pair of symmetric zeros holds $Z(u,\,v,\,b,\,c)=X(u,\,vb,\,c)$. Next, take $X(u,\,v,\,b,\,c)\in\St\!\Sp_{2n}(B,\,I)$ and take $g\in\St\!\Sp_{2n}(B)$ such that $\phi(g)u=e_1$. Then $\kappa$ sends it to $X(u,\,vb,\,c)$ and $T\circ\lambda_a^*$ to $$Z\Big(\lambda_a(u)a^N,\,\lambda_a(v)a^{N+M},\,\frac{b}{a^{2N+M}},\,\frac{c}{a^{2N}}\Big).$$ Now, conjugate both elements by $g$ and use the previous observation to show that they coincide. The lower triangle can be treated similarly.

Finally, we show that $\mathrm T$ maps 
$$g\in\mathrm{Im}\Big(\St\!\Sp_{2n-2}(B_a,I_a)\rightarrow\St\!\Sp_{2n}(B_a,I_a)\Big)$$ 
to the element of $\mathrm{Im}\Big(\St\!\Sp_{2n-2}(B)\rightarrow\St\!\Sp_{2n}(B)\Big)$. For $n=2$ there exist no obvious analogue of the Tulenbaev map, so that for $n=3$ we argue as follows (for $n>3$ this argument works as well). We can assume that $g=X(u,\,v,\,b,\,c)$ for $u$ and $v$ such that $u_n=u_{-n}=v_n=v_{-n}=0$. Then, consider lifts $\tilde u$, $\tilde v$ of $ua^N$ and $va^N$. Their $\,\pm n$-th coordinates localise to zeros, thus increasing $N$ we may assume that they actually are zeros. Thus, it remains to show that for $u$ and $v$ such that $u_n=u_{-n}=v_n=v_{-n}=0$ one has $Z(u,\,v,\,b,\,c)\in\mathrm{Im}\St\!\Sp_{2n-2}(B)$. As above, in this situation $Z(u,\,v,\,b,\,c)=X(u,\,vb,\,c)$. By Lemma~\ref{xlist}\,(X6), we only need to consider $X(u,\,0,\,c)$ with $u_n=u_{-n}=0$. With the use of Lemma~\ref{ylist}\,(Y3--Y6) we reduce it to the case $X(e_i,\,v,\,c)$ with $i\neq\pm n$, $v_n=v_{-n}=0$. To conclude the proof, decompose this element by Lemma~\ref{xlist2}\,(X10) as a product of usual elementary generators of Steinberg group
$$X(e_i,\,v,\,c)=X(e_i,\,0,\,c-\sum v_kv_{-k})\prod X(e_i,\,e_kv_k,\,0).$$
Al of them lie in $\mathrm{Im}\St\!\Sp_{2n-2}(B)$.
\end{proof}

\end{document}